\documentclass[11pt,reqno]{amsart}
\usepackage{a4wide}
\usepackage{amsmath, amsthm, amssymb}
\usepackage{epsfig}
\usepackage{array}
\usepackage{amsmath}
\usepackage{amsfonts}
\usepackage{amsrefs}
\usepackage{amsmath}%
\usepackage{amsfonts}%
\usepackage{amssymb}%
\usepackage{graphicx}
\usepackage{hyperref}
\usepackage{xcolor}
\usepackage{amsfonts}
\usepackage{amssymb}
\usepackage{amsmath,amsxtra,amssymb}

\numberwithin{equation}{section}

\newtheorem{theorem}{Theorem}[section]
\newtheorem{lemma}[theorem]{Lemma}
\newtheorem{remark}[theorem]{Remark}

\newtheorem{proposition}[theorem]{Proposition}
\newtheorem{definition}[theorem]{Definition}

\newcommand{\st}{\;\vline\;}
\newcommand{\G}{\mathbb G}
\renewcommand{\H}{\mathbb H}
\newcommand{\field}[1]{\mathbb{#1}}

\newcommand{\C}{{\mathbb{C}}}

\newcommand{\id}{\mathrm{id}}
\newcommand{\ot}{{\, \otimes \, }}
\newcommand{\om}{{\omega}}

\newcommand{\vtp}{{\,\overline{\otimes}\,}}

\newcommand{\B}{{\mathcal{B}}}

\newcommand{\act}{{\curvearrowright}}

\newcommand{\sC}{\mathsf{C}}

\newcommand{\LL}{{L^{\infty}(\G)}}
\newcommand{\LO}{{{L}^{1}(\G)}}
\newcommand{\LT}{{L^{2}(\G)}}

\newcommand{\DD}{{{\ell}^{\infty}(\G)}}

\newcommand{\DO}{{{\ell}^{1}(\G)}}

\newcommand{\LLL}{{L^{\infty}(\widehat{\mathbb{G}})}}

\newcommand{\vv}{\mathrm{V}}
\newcommand{\ww}{\mathrm{W}}

\newcommand{\CC}{{\field{C}}}
\newcommand{\cst}{\ifmmode\mathrm{C}^*\else{$\mathrm{C}^*$}\fi}
\newcommand{\sA}{\mathsf{A}}
\newcommand{\hh}[1]{\widehat{#1}}
\newcommand{\FO}{\mathbb{F}O}

\newcommand{\h}{\mathbb H}

\DeclareMathOperator{\irr}{Irr}

\allowdisplaybreaks

\allowdisplaybreaks

\title
[]
{SAT actions of discrete quantum groups and minimal injective extensions of their von Neumann algebras}

\author[M. Kalantar]{Mehrdad Kalantar}
\address{Mehrdad Kalantar\\ University of Houston\\ USA}
\email{mkalantar@uh.edu}

\author [F. Khosravi]{Fatemeh Khosravi}
\address{Fatemeh Khosravi\\ Seoul National University\\ South Korea}
\address{Current address: University of Isfahan\\ Iran}
\email{f.khosravi@mcs.ui.ac.ir}

\author [M. S. M. Moakhar]{Mohammad S. M. Moakhar}
\address{Mohammad S. M. Moakhar\\ Institute for Research in Fundamental Sciences (IPM)\\Iran}
\email{m.mojahedi@ipm.ir}

\thanks{{\ }\\[-3ex]MK was partially supported by the Simons Foundation Collaboration Grant \#713667 and by the NSF Grant DMS-2155162. \\FK was supported by the National Research Foundation of Korea (NRF) grant funded by the government of Korea (MSIT) (No. 2017R1E1A1A03070510 and 2020R1C1C1A01009681)}

\begin{document}

\begin{abstract}
We introduce a natural generalization of the notion of strongly approximately transitive (SAT) states for actions of locally compact quantum groups. 
In the case of discrete quantum groups of Kac type, we show that the existence of unique stationary SAT states entails rigidity results concerning injective extensions of quantum group von Neumann algebras. 
\end{abstract}

\maketitle

\section{Introduction}\label{sect1} 

The theory of boundary actions in the sense of Furstenberg \cite{Furs} has been a central concept in ergodic theory of non-amenable groups in the past few decades.
Several components of the theory have been imported to the setting of locally compact quantum groups. 
In \cite{I}, Izumi introduced and studied noncommutative Poisson boundaries of discrete quantum groups. 
This initiated a whole body of work on concrete realization and applications of these boundary actions (see e.g. \cites{IzuNeshTus06,VV,Tom07,VaesVander08,VaesVander10,KNR2}).

Following the recent striking applications of the Furstenberg boundary actions (the topological counterpart of the Poisson boundary) in certain problems in $C^*$-algebra theory of discrete groups, the notion was extended to the discrete quantum setting in \cite{KKSV}.

Similarly to the commutative case, both concepts of noncommutative boundaries mentioned above have proven to have important applications in ergodic theory of discrete quantum groups and the structure theory of their operator algebras.

The defining feature of boundary actions is the {\it proximality} property. Roughly speaking, in the measurable setting (e.g. the Poisson boundary) this means contractibility of the boundary measure (state) along paths of the underlying random walk.
A closer look at many applications of boundary actions reveals that very often it is the contractibility property itself that yields rigidity properties associated with boundary actions.
This is a point that we would like to emphasize in this work. 
In the group setting, the notion of contractibility in the measurable setup was formalized by Jaworski in \cite{Jaw} under the term {\it strongly approximately transitive (SAT)} actions. 
In this work, we introduce a natural generalization of SAT actions in the setting of actions of locally compact quantum groups $\G$ on von Neumann algebras $N$.
The primary examples of such states, which are also our main interests, are the `Poisson boundary states' of discrete quantum groups $\G$, namely, the restriction of the co-unit of $\G$ to the space of $\mu$-harmonic elements in $\ell^\infty(\G)$, where $\mu$ is a normal state on $\ell^\infty(\G)$ (see Section~\ref{sec:SAT} for more details).

In the commutative case, there are examples of $\mu$-stationary SAT actions that are not doubly-ergodic \cite{Creu}, hence not $\mu$-boundaries \cite{Kai}. But to the best of our knowledge, those are the only known examples. This shows how close the notions are in general. Despite this, the operator theoretic nature of SAT actions offers much more flexibility, especially in the quantum setting. Moreover, SAT property can be described in commutative terms: it is the irreducibility of the action of the semigroup of ``quantum probabilities'' on $\G$ on the normal state space of the von Neumann algebra $N$. Thus, standard operator theoretic and dynamical tools could be effective in this context.
In addition, this notion makes sense for general normal states of $N$ and does not require stationarity. This could be particularly advantageous since requiring the existence of normal stationary states on noncommutative von Neumann algebras is often too restrictive.


The main result of this paper (Theorem~\ref{main}) is a rigidity result for injective extensions of von Neumann crossed products of Poisson boundary actions of certain discrete quantum groups, (given an inclusion $M\subseteq N$ of von Neumann algebras with $N$ injective, we say $N$ is an injective extension of $M$). This is inspired by the recent work of Hartman and the first-named author in \cite{HK}, where similar results are proven for actions of locally compact groups on their boundaries. 
The key ingredient of the proof of Theorem~\ref{main} is the fact that the `boundary state' which is always SAT, in many concrete cases, is also unique stationary. 

Examples of uniquely stationary models of noncommutative boundaries for large classes of discrete quantum groups have been proven in recent work \cites{KKSV, HHN22}.  
However, the uniqueness conditions proven in op-cit are different. More precisely, in \cite{KKSV} it is proved that the boundary state on the Gromov boundary of an orthogonal free discrete quantum group $\G$ is the unique $\mu$-stationary state for any generating $\mu\in\ell^1(\G)$. In contrast, in the more general setting of Kac type considered in \cite{HHN22}, the boundary state is proved to be the unique $\mu$-stationary state which is also $\hh\G$-invariant.
This, in our main result, translates to maximal injectivity in the category of von Neumann algebras equipped with both $\G$ and $\hh\G$ actions, or equivalently an action of the Yetter--Drinfeld algebra $D(\G)$.

To unify these setups, we define the new concept of \emph{relative Yetter--Drinfeld algebras} associated to a given closed quantum subgroup $\hh\H$ of $\hh\G$. 
This notion provides the general framework in which we prove our main theorem, and the two extreme cases, where $\hh\H$ is trivial or $\hh\G$ correspond to the above two classes of examples mentioned above.

The paper is structured as follows. In addition to this introduction, the paper consists of five more sections. In Section~\ref{sect2} we give a brief review of basic definitions and some basic facts concerning quantum groups and their actions.
In Section~\ref{sec3} we introduce the concept of relative Yetter--Drinfeld algebras, and prove several lemmas about them that we need for our main results. Although this concept is certainly of independent interest and can lead to new interesting examples of locally compact quantum groups, we restrict ourselves to only the facts we require for the purpose of our main results.
In Section~\ref{sec:SAT}, we introduce SAT actions of locally compact quantum groups.
We generalize several main properties of SAT actions of locally compact groups to the quantum setting.
Section~\ref{uss} contains our main general result Theorem~\ref{main} where we prove for $(\G,{\H})$-Yetter--Drinfeld algebras $N$ satisfying some extra conditions, the inclusion $\LLL\subset \G\ltimes_{\alpha} N$ does not admit any proper ${\hh{\mathbb{H}}}$-injective intermediate von Neumann subalgebras. 
The main applications of Theorem~\ref{main} are proven in Section~\ref{ex}. 
 \\

\noindent
\textbf{Acknowledgements.}
We are grateful to the anonymous referee for their insightful comments and suggestions which, in particular, resulted in substantial improvement of the results and the presentation of the paper.
We also thank Sergey Neshveyev and Adam Skalski for their helpful comments.

\section{Preliminaries}\label{sect2}
In this section, we establish our notation, and briefly review some basic definitions and results concerning locally compact quantum groups and their actions. For more details about the theory of locally compact quantum groups, we refer the reader to \cites{KV,KVvN}. Most parts of this paper are focused on discrete quantum groups, and all the facts we use concerning discrete and compact quantum group may be found in \cites{I, Ne, VV}.

\subsection*{Locally compact quantum groups}
A von Neumann algebraic locally compact quantum (lcq) group is a quadruple $\mathbb{G} = (\LL, \Delta^\G,\varphi^\G,\psi^\G)$, where $\LL$ is a von Neumann algebra with a coassociative comultiplication $\Delta^\G\colon\LL\to\LL\vtp\LL$, and $\varphi^\G$ and $\psi^\G$ are, respectively, normal semifinite faithful (n.s.f.) left and right Haar weights on $\LL$. 
The GNS Hilbert space of the left Haar weight $\varphi^\G$ will be denoted by  $\LT$ and we write $\LO$ for the predual of the von Neumann algebra $\LL$. 

We denote by $\ww^\G,\vv^\G\in B(\LT\otimes \LT)$ the left and right  multiplicative unitaries of $\G$, respectively. They implement the comultiplication: $\Delta^\G(x)=\vv^\G(x\otimes 1)\vv^{\G^*}=\ww^{\G^*}(1\otimes x)\ww^\G$ for $x\in\LL$. Moreover,
$$\LL =\bigl\{ (\omega\otimes\id)\vv^\G\st\omega\in B(\LT)_*\bigr\} ^{\prime\prime}=\bigl\{ (\id\otimes\omega)\ww^\G\st\omega\in B(\LT)_*\bigr\}^{\prime\prime} .$$

The norm closure $\bigl\{ (\id\otimes\omega)\ww^\G\st\omega\in B(\LT)_*\bigr\}^{\|\cdot\|} \subseteq \LL$ is denoted by $C_0(\G)$. The restriction of $\Delta^\G$ to the $C^*$-algebra $C_0(\G)$ defines a comultiplication $\Delta^\G: C_0(\G)\to M(C_0(\G)\ot_{\rm min}C_0(\G))$, turning it into a \emph{$C^*$-algebraic locally compact quantum group}.

To every lcq group $\G$ there also associates a \emph{universal} $C^*$-algebra $C_0^u(\G)$, which is a $C^*$-algebra extension of $C_0(\G)$, and is endowed with a comultiplication $\Delta^{\G^u}: C_0^u(\G)\to M(C_0^u(\G)\ot_{\rm min}C_0^u(\G))$, 
and denoting by $\Lambda_\G: C_0^u(\G)\to C_0(\G)$ the \emph{canonical} surjective $*$-homomorphisms, we have $\Delta^\G\circ\Lambda_\G = (\Lambda_\G\ot \Lambda_\G)\circ  \Delta^{\G^u}$. 

%
The dual of a lcq group $\G$, will be denoted by $\hh\G$. We have
$\LLL=\bigl\{( \id\otimes\omega)\ww^{\hh\G}\st\omega\in B(\LT)_*\bigr\}^{\prime\prime}$, $\Delta^{\hh\G}(x)=(\ww^{\hh\G})^*(1\otimes \hh x){\ww^{\hh\G}}$ for $\hh x\in\LLL$,
where $\ww^{\hh\G}:=\sigma({\ww^{\G}})^* $ and $\sigma$ is the flip operator. We have $\ww^\G\in\LL\vtp\LLL$ and $\vv^\G\in L^\infty(\hh\G)^\prime\vtp\LL$.

The opposite quantum group $\G^{\text {op}}$ of $\G$ is defined by $L^\infty(\G^{\text {op}})=\LL$, and $\Delta^{\G^{\text {op}}}(\cdot)=\sigma\circ\Delta^\G(\cdot)$, and the commutant quantum group $\G^\prime$ by $L^\infty(\G^{\prime})=\LL^\prime$, with $\Delta^{\G^\prime}(\cdot)=(J^\G\otimes J^\G)(\Delta^\G( J^\G  \cdot J^\G )) (J^\G \otimes J^\G)$, where $J^\G$ is the modular conjugation of $\varphi^\G$ (see \cite{KVvN} for more details).
We have a canonical identification $\hh{\G^{\text {op}}}\cong {\hh\G}'$.

When $\G$ is commutative, i.e. a locally compact group $G$, then $\LL=L^\infty(G)$, $\varphi^\G$ and $\psi^\G$ are the left and right Haar integrals, respectively, and $C_0^u(\G) = C_0(\G) = C_0(G)$ is the $C^*$-algebra of continuous functions on $G$ vanishing at infinity.
We have $\LLL=VN(G)$ is the (left) group von Neumann algebra of $G$ and $C_0(\hh\G) = C_r^*(G)$ is the reduced $C^*$-algebra of $G$, and $C_0^u(\hh\G) = C^*(G)$ is the full $C^*$-algebra of $G$.

A lcq group $\G$ is said to be \emph{compact} if the Haar weights are finite, which is equivalent to the $C^*$-algebra $C_0(\G)$ being unital. We say $\G$ is \emph{discrete} if its dual $\hh\G$ is compact. 
A compact quantum group $\G$ is of \emph{Kac type} if $\varphi^\G$ is a trace, and a discrete quantum group $\G$ is of {Kac type} if $\hh\G$ is of {Kac type}.


\subsection*{Actions of lcq groups} 
For general theory of $C^*$ and von Neumann algebraic actions of lcq groups $\G$ and their associated crossed product constructions, 
we refer the reader to \cites{Ind, V}. 
Also, more details on the facts we use regarding actions of discrete quantum groups can be found in \cite[Section 2.5]{KKSV}. We briefly recall some definitions and basic facts.

Let $\G$ be a lcq group. A \emph{left action} of $\G$ on a 
$\cst$-algebra $\sA$ is an injective morphism $\alpha: \sA\to M(\C_0(\G)\otimes \sA)$ such that $(\Delta^\G\otimes\id)\,\alpha=(\id\otimes\alpha)\,\alpha$. The action $\alpha$ is said to be continuous if $\alpha(\sA)(C_0(\G)\otimes 1) = C_0(\G) \otimes \sA$. We say $\sA$ is a (left) $\G$-$C^*$-algebra if it is equipped with a continuous left action $\alpha$ of $\G$. 

If $\sA$ is a unital $\G$-$C^*$-algebra, then 
for any 
$a\in\sA$ and $\om\in \C_0(\G)^*$, we have $(\om \otimes \id)\alpha(a)\in \sA$; indeed, for every $x\in C_0(\G)$ we have $\alpha(a)(x\otimes1)\in C_0(\G) \otimes \sA$, and so $((x\cdot\om^\prime) \otimes \id)\alpha(a)\in \sA$ for every $\om^\prime\in \C_0(\G)^*$, where $x\cdot\om^\prime\in \C_0(\G)^*$ is defined by $x\cdot\om^\prime(y) = \om^\prime(yx)$ for all $y\in \C_0(\G)$. The set $\{x\cdot\om^\prime \colon x\in C_0(\G), \om^\prime\in \C_0(\G)^*\}$ spans a norm-dense subspace of $\C_0(\G)^*$ (see e.g. either the proof of \cite[Result 3.4]{KV} or \cite[Theorem 2.4]{Runde09}), hence the claim follows.

It turns out that any action of a discrete quantum group $\G$ on a unital $C^*$-algebra $\sA$ is automatically continuous (see e.g. the last part of the proof of \cite[Theorem 4.9]{KKSV}).

A {left action} of $\G$ on a von Neumann algebra $N$ is a unital injective normal $*$-homomorphism $\alpha:N\to\LL\vtp N$ such that $(\Delta^\G\otimes\id)\,\alpha=(\id\otimes\alpha)\,\alpha$. In this case, we say $N$ is a $\G$-von Neumann algebra.
The right actions of lcq groups on $C^*$-algebras and von Neumann algebras are defined similarly. 

Let $N$ be a $\G$-von Neumann algebra with the associated action $\alpha$. A von Neumann subalgebra $M\subseteq N$ is said to be {\it $\G$-invariant} if $\alpha(M) \subseteq \LL\vtp M$; in this case, $M$ is a $\G$-von Neumann algebra with respect to the restriction of $\alpha$ to $M$.

Similarly, a $C^*$-subalgebra $\sA\subseteq N$ is said to be {\it $\G$-invariant} if $\sA$ is a $\G$-$C^*$-algebra with respect to the  restriction of $\alpha$ to $\sA$. It should be noted that $M(\C_0(\G)\otimes \sA)\subset \LL\vtp N$, thanks to the following general fact: if $\sC$ is a $C^*$-subalgebra of a von Neumann algebra $P$, then $P$ contains the multiplier algebra $M(\sC)$ of $\sC$ (see e.g. Section 3.12 of \cite{Ped-book}, in particular Proposition~3.12.3 in loc. cit.). 


A weight $\tau$ on $N$ is said to be $\alpha$-invariant if $\tau\left( (\omega\otimes\id)\alpha(x)\right) = \tau(x)$ for every normal state $\omega$ on $N$ and every positive $x\in N$ with $\tau(x)<\infty$.

The crossed product  of the action $\alpha$ of $\G$ on a von-Neumann algebra $N$ is defined by
$$
\G\ltimes_{\alpha}N=\{\alpha(N)(\LLL\vtp \C) \}^{''}\subseteq B(\LT)\otimes N. 
$$
It admits a left $\G$-action 
\begin{equation}\label{beta}
\tilde\alpha: \G\ltimes_{\alpha}N \to \LL \vtp \G\ltimes_{\alpha}N, \qquad \tilde\alpha(z)=(\ww^{\G}_{12})^*z_{23}\ww^{\G}_{12} 
\end{equation}
and a right $\hh\G$-action
\begin{equation}\label{act2}
\widehat{\alpha}:\G\ltimes_{\alpha}N\to \G\ltimes_{\alpha}N\vtp\LLL,\qquad \widehat{\alpha}(z)= \vv^{\hh\G}_{13}z_{12}(\vv^{\hh\G})^*_{13} .
\end{equation}
The restriction of $\tilde\alpha$ to $\alpha(N)$ is $(\Delta\otimes\id)= (\id\otimes\alpha)$, the restriction of $\widehat{\alpha}$ to $\alpha(N)$ is the trivial action, and the restriction of $\hh\alpha$ to $L^\infty(\hh\G)\vtp \C$ is equal to $\hh\Delta(\cdot)_{13}$.

If $\G$ is a discrete quantum group, then the map $E=(\id\otimes \varphi^{\hh\G})\circ\widehat{\alpha}$ defines a faithful normal conditional expectation from $\G\ltimes_{\alpha}N$ onto $\alpha(N)$. 
We refer to $E$ as the \emph{canonical conditional expectation}.

Given actions $\alpha$ and $\beta$ of a lcq group $\G$ on von Neumann algebras $M$ and $N$, respectively, we say a map $\Phi:N\to M$ is $\G$-\emph{equivariant} if $(\id\otimes\Phi)\,\alpha=\beta\,\Phi$, or equivalently, $\Phi\,(\mu\otimes\id)\,\alpha=(\mu\otimes\id)\,\beta\,\Phi$ for all $\mu\in\LO$.

The following fact is used in several proofs below, it generalizes part of \cite[Lemma 5.2]{KKSV}.

\begin{lemma}\label{lem:cond-eq}
Let $\G$ be a discrete quantum group, and let $N$ be a $\G$-von Neumann algebra. The canonical conditional expectation $E:\G\ltimes_{\alpha}N\to \alpha(N)$ is $\G$-equivariant if and only if $\G$ is of Kac type.
\end{lemma}
\begin{proof}
The restriction of $E$ to $\alpha(N)$ is the identity, hence $\G$-equivariant. Since $\alpha(N)$ is in the multiplicative domain of $E$, and $\G\ltimes_{\alpha}N$ is generated by $\alpha(N)$ and $L^\infty(\hh\G)\otimes \C$, we need to show that the restriction of $E$ to $L^\infty(\hh\G)\otimes \C$ is $\G$-equivariant if and only if $\G$ is of Kac type.
For this, let $y\in L^\infty(\hh\G)$ and compute
\begin{align*}
(\id\otimes E)\tilde{\alpha}(y\otimes 1)
&=
(\id\otimes E)\Big(\big((\ww^\G)^*(1\otimes y)\ww^\G\big)  \otimes 1\Big)
\\&=
(\id\otimes \varphi^{\hh\G})(\id\otimes\hh\alpha) \Big(\big((\ww^\G)^*(1\otimes y)\ww^\G\big) \otimes 1\Big)
\\&=
(\id\otimes \varphi^{\hh\G})\Big( (\id\otimes \hh\Delta)  \big((\ww^\G)^*(1\otimes y)\ww^\G\big) \Big)_{124}
\\&=
 \Big( (\id\otimes \varphi^{\hh\G})  \big((\ww^\G)^*(1\otimes y)\ww^\G \big) \Big)\otimes1\otimes1,
\end{align*}
where in the last equality we used the left invariance of $\varphi^{\hh\G}$.
On the other hand, 
\begin{align*}
{\alpha}\big(E(y\otimes 1)\big)
=
{\alpha}\big((\id\otimes\varphi^{\hh\G})\Delta^{\hh\G}(y)_{13}\big) 
= 
\varphi^{\hh\G}(y) ,
\end{align*}
which shows that $E$ is $\G$-equivariant if and only if $\varphi^{\hh\G}$ is $\G$-invariant. Invoking \cite[Lemma 5.2]{KKSV} yields the result.
\end{proof}
In several results, we use the notion of unitary implementation of an action in the sense of \cite{V}. We recall that given an action $\alpha$ of a lcq group $\G$ on a von Neumann algebra $N$, and an n.s.f. weight $\theta$ on $N$, the unitary implementation of the action is a unitary corepresentation $U_\theta\in B(L^2(\G)\otimes L^2(N, \theta))$ of $\G$ such that $\alpha( \cdot ) = U_\theta (1\otimes \cdot) U_\theta^*$.
Then the formula $\alpha'(\cdot):=U_\theta^*(1\otimes \cdot) U_\theta$ defines an action of $\G^{\text{op}}$ on the commutant von Neumann algebra $N'\subseteq B(L^2(N, \theta))$ \cite[Proposition 6.8]{ENOCK}. 

The following theorem is proved in \cite[Theorem 11.7 (ii)]{ENOCK} for actions of measured quantum groupoids. We include the proof in the case of actions of locally compact quantum groups for the convenience of the reader.

\begin{theorem}\label{commutant}
Let $\alpha : \G\curvearrowright N$ be an action of a lcq group $\G$ on a von Neumann algebra $N$. Let $\theta$ be an n.s.f. weight on $N$, and $U_\theta\in B(L^2(\G)\otimes L^2(N, \theta))$ the unitary implementation of $\alpha$. Then
\[
(\G\ltimes_\alpha N)'	= U_\theta(\G^{\text{op}}\ltimes_{\alpha'} N')U_\theta^*.
\]
\end{theorem}
\begin{proof}
By \cite[Theorem 2.6]{V} there is an action $\gamma$ of $\G$ on the von Neumann algebra $B(L^2(\G))\vtp N$ defined by
$\gamma:z\mapsto 
(\chi\otimes\id)\big(V^*_{12} ((\id\otimes\alpha)(z))V_{12}\big)$,
and moreover, we have
\[
\G\ltimes_\alpha N =\{z\in B(L^2(\G))\vtp N \st \gamma(z)=1\otimes z\} .
\]
Thus, using the fact that $U_\theta$ implements $\alpha$, we conclude that
\begin{equation}\label{Saturated}
\G\ltimes_\alpha N =\{z\in B(L^2(\G))\vtp N \st {(U^*_\theta)}_{23} V_{12}\, z_{13}= z_{13}\,{(U^*_\theta)}_{23} V_{12}\} .
\end{equation}
By \cite[Theorem 4.4]{V} the unitary $U_\theta$ is a corepresentation: 
$V_{12}\,{(U_\theta)}_{13}\,V_{12}^*= {(U_\theta)}_{23}{(U_\theta)}_{13}$. 
Applying this to \eqref{Saturated}, we get
\[\begin{split}
&\G\ltimes_\alpha N 
=\left\{z\in B(L^2(\G))\vtp N \st V_{12} {(U^*_\theta z U_\theta)}_{13}={(U^*_\theta z U_\theta)}_{13} V_{12}\right\} .
\end{split}\]
In particular, it follows that $\G\ltimes_\alpha N$ contains exactly those $z\in B(L^2(\G))\vtp N$ for which $(\id\otimes\rho)\left({U^*_\theta} z {U_\theta}\right)$ commutes with $(\id\otimes\om)V$ for all $\rho\in N_*$ and $\om\in B(L^2(\G))_*$. Since $L^\infty(\widehat\G)' = \{(\id\otimes\om)V \st \om\in B(L^2(\G))_*\}^{\prime\prime}$, 
we conclude
\[\begin{split}
(\G\ltimes_\alpha N)^\prime &= 
\Big((B(L^2(\G))\vtp N) \,\cap\, (U_\theta(\LLL\vtp B(H_\theta))U^*_\theta)\Big)^\prime
\\&=
\big\{(1\vtp N') \, \cup \, (U_\theta (\LLL' \vtp 1) U^*_\theta)\big\}^{\prime\prime}
\\&=
U_\theta\,\big\{(U^*_\theta(1\vtp N')U_\theta) \, \cup \, (\LLL' \vtp 1) \big\}^{\prime\prime}\,U^*_\theta
\\&=
U_\theta\,\{\alpha^\prime(N^\prime)\, \cup \, (L^\infty(\widehat\G)^\prime \vtp 1)\}^{\prime\prime}\,U^*_\theta
= U_\theta\,(\G^{\text{op}}\ltimes_{\alpha'} N')\,U^*_\theta ,
\end{split}\]
where in the last equality we use the fact that $L^\infty(\widehat{\G^{\text{op}}}) = L^\infty(\hh\G)^\prime$.
\end{proof}



\section{Relative Yetter--Drinfeld algebras}\label{sec3}
In this section we introduce the notion of relative Yetter--Drinfeld algebras, which provides a unified framework for our main results.

Let $\G$ be a lcq group, and let $\hh{\mathbb{H}}$ be a closed quantum subgroup of $\hh\G$. This means there is an injective normal $*$-homomorphism $\gamma: L^\infty(\mathbb{H})\to L^\infty(\G)$ satisfying $(\gamma \otimes \gamma) \circ \Delta^{\H} = \Delta^{\G} \circ \gamma$. 
The associated \emph{bicharacter} is the unitary $\overline{\ww}^{\mathbb{H}}=(\gamma\otimes\id)\ww^{\mathbb{H}}\in L^\infty(\G)\vtp L^\infty(\hh{\mathbb{H}})$ (see \cite[Definition~3.1]{MRW12}), and the map $\rho: L^\infty(\hh\G)\to L^\infty(\hh\h)\vtp L^\infty(\hh\G)$ defined by $\rho(\hh x):=(\overline{\ww}^{\hh\h})^*\left( 1\otimes \hh x\right) \overline{\ww}^{\hh\h}$, where $\overline{\ww}^{\hh\h}:= (\sigma\overline{\ww}^{\h})^*$, is the canonical left action of $\hh\h$ on $L^\infty(\hh\G)$. 

By \cite[Proposition 1.8]{DKSS}, the unitary $\overline{\ww}^{\hh\h}$ is also a bicharacter (from $\hh\H$ to $\hh\G$), i.e. satisfying
\begin{equation}\label{bi-ch}
(\id\otimes \Delta^\G)\overline{\ww}^{\hh\h}=\overline{\ww}^{\hh\h}_{13} \overline{\ww}^{\hh\h}_{12} .
\end{equation}
Note that the convention used in \cite{DKSS} is different from us, the authors there consider right multiplicative unitaries, and the dual quantum group in their setup is the commutant of the dual for us.

There is a non-degenerate $*$-homomorphism $\pi_{\hh\h}: C_0^u(\hh\G)\to M(C_0^u(\hh\H))$ 
intertwining the respective coproducts and such that 
\begin{equation}\label{eqq}
\rho(\Lambda_{\hh\G}(\hh a)) = ((\Lambda_{\hh\h}\circ\pi_{\hh\h})\otimes \Lambda_{\hh\G})\Delta^u_{\hh\G}(\hh a) ,
\end{equation}
for every $\hh a \in C_0^u(\hh\G)$ \cite[Section 1.3]{DKSS}.
\begin{definition}\label{defYD}
Let $\G$ be a locally compact quantum group and let $\hh{\mathbb{H}}$ be a closed quantum subgroup of $\hh\G$.
 A $(\G,\mathbb{H})$-Yetter-Drinfeld $C^*$-algebra is a $C^*$-algebra $A$ equipped with a continuous left action $\alpha$ of $\G$ and a continuous left action $\lambda$ of $\hh{\mathbb{H}}$ such that 
\begin{equation}\label{YD}
\overline{\ww}^{\mathbb{H}}_{12}\big((\id\otimes\lambda) \alpha(a)\big)(\overline{\ww}^{\mathbb{H}}_{12})^* = (\sigma\otimes \id)(\id\otimes\alpha) \lambda(a) .
\end{equation}
for all $a\in A$. Similarly, a von Neumann algebra $M$ equipped with left actions $\alpha$ of $\G$ and $\lambda$ of $\hh{\mathbb{H}}$ satisfying \eqref{YD} is called a $(\G,\mathbb{H})$-Yetter-Drinfeld von Neumann algebra.
\end{definition}

\begin{remark}{\normalfont
(i)\ When $\hh{\mathbb{H}}$ is the trivial quantum subgroup, then $\overline{\ww}^{\mathbb{H}} = 1$ and a $(\G,\mathbb{H})$-Yetter Drinfeld $C^*$-algebra is just a $\G$-$C^*$-algebra.
\\[0.2ex]
(ii)\ When $\hh{\mathbb{H}}=\hh\G$, then $\overline{\ww}^{\mathbb{H}}=\ww^\G$ and $(\G,\mathbb{H})$-Yetter--Drinfeld $C^*$-algebras are exactly $\G$-Yetter Drinfeld $C^*$-algebras.
}
\end{remark}

\begin{remark}{\normalfont
Given a lcq group $\G$ and a closed quantum subgroup $\hh{\mathbb{H}}$ of $\hh\G$, just similar to the case of usual Yetter--Drinfeld algebras, one can define the relative Drinfeld double $D(\G,{\H})$ to be the lcq group defined by $C_0(D(\G,{\H}))= C_0(\G)\otimes C_0(\hh{\mathbb{H}})$ with the comultiplication
\[
\Delta^{D(\G,{\H})}=(\id\otimes \sigma \otimes \id)\big(\overline{\ww}^{{\H}}_{23}(\Delta^\G\otimes \Delta^{\hh{\mathbb{H}}})(\overline{\ww}^{{\H}}_{23})^*\big).
\]
Then, slightly modifying the proof of \cite[Proposition 3.2]{NV}, it follows that a $(\G,{\H})$-Yetter--Drinfeld $C^*$-algebra is the same thing as a $D(\G,{\H})$-$C^*$-algebra.
Since we will not need any of these facts, we will not go into further details of these facts.}
\end{remark}

\begin{lemma}\label{lem:GHcrossed}
Let $\G$ be a lcq group and let $\hh\h$ be a closed quantum subgroup of $\hh\G$, and let $M$ be a $(\G,\H)$-Yetter-Drinfeld von Neumann algebra equipped with left actions $\alpha$ of $\G$ and $\lambda$ of $\hh\h$.
Then the formula
\[
\tilde{\lambda}(z)= (\sigma\otimes\id)\big(\overline{\ww}^{\h}_{12} \left( \left(\id\otimes \lambda\right) z \right) (\overline{\ww}^{\h}_{12})^*\big)
\]
defines a left action of $\hh\h$ on $\G\ltimes_\alpha M$, and equipped also with the left action $\tilde{\alpha}$ of $\G$ defined in \eqref{beta}, $\G\ltimes_\alpha M$ turns into a $(\G,{\H})$-Yetter-Drinfeld von Neumann algebra.
\begin{proof}
For every $\hh x\in L^\infty(\hh\G)$, 
we have
\begin{equation}\label{eq*}
\tilde{\lambda}(\hh x\otimes 1) = (\sigma\otimes\id)\left(\big(\overline{\ww}^{\h} \left( \hh x\otimes 1 \right) (\overline{\ww}^{\h})^*\big)\otimes 1\right)
=
\rho(\hh x)\otimes 1 ,
\end{equation}
where $\rho$ 
is the left action of $\hh\h$ on $L^\infty(\hh\G)$ recalled above.
It follows
\[\begin{split}
(\Delta^{\hh\h}\otimes\id)\tilde{\lambda}(\hh x\otimes 1) 
&=
\big((\Delta^{\hh\h}\otimes\id)(\rho(\hh x)\big)\otimes 1 
=
\big((\id\otimes\rho)\rho(\hh x)\big)\otimes 1
\\&=
(\id\otimes\tilde{\lambda})\big(\rho(\hh x)\otimes 1\big)
=
(\id\otimes\tilde{\lambda})\tilde{\lambda}(\hh x\otimes 1)
,
\end{split}\]
and so $\tilde{\lambda}$ defines a left action of $\hh\h$ when restricted on $L^\infty(\hh\G)\otimes \C$.
From \eqref{YD} we have
\begin{equation}\label{eq:alf-lamb}
(\id\otimes\alpha) \lambda(a) =\tilde{\lambda}\,(\alpha(a))
\end{equation}
for every $a\in M$, and therefore
\begin{align*}
(\Delta^{\hh\h}\otimes\id)\tilde{\lambda}\,(\alpha(a))
&= (\Delta^{\hh\h}\otimes\id)\, (\id\otimes\alpha) \lambda(a)
=(\id\otimes\id\otimes\alpha)\,(\Delta^{\hh\h}\otimes\id)\,\lambda(a)\\
&=(\id\otimes\id\otimes\alpha)\,(\id\otimes\lambda)\,\lambda(a)
=(\id\otimes \tilde{\lambda}\alpha)\,\lambda(a)
=(\id\otimes \tilde{\lambda})\,\tilde{\lambda}\,(\alpha(a)) \,,
\end{align*}
where \eqref{eq:alf-lamb} is used in the first, forth and fifth equalities, and the fact that $\lambda$ is a left action of $\hh\h$ is used in the third equality.
This implies that the restriction of $\tilde{\lambda}$ to $\alpha(M)$ yields a left action of $\hh\h$ on $\alpha(M)$.
Since the von Neumann $\G\ltimes_\alpha M$ is generated by $L^\infty(\hh\G)\otimes \C$ and $\alpha(M)$, it follows that $\tilde\lambda$ is a left action of $\hh\h$ on $\G\ltimes_\alpha M$.

We next check \eqref{YD} for $\tilde{\lambda}$. For every $a\in M$, 
\begin{align*}
(\sigma\otimes\id)\,(\id\otimes \tilde{\alpha})\, \tilde{\lambda} (\alpha(a))
&= (\sigma\otimes\id)\,(\id\otimes \tilde{\alpha})\, (\id\otimes\alpha)\,\lambda(a)&&\text{by \eqref{eq:alf-lamb}}\\
&=(\sigma\otimes\id)\,(\id\otimes \id \otimes \alpha)\, (\id\otimes\alpha)\,\lambda(a)&&\text{}\\
&=(\id\otimes \id \otimes \alpha)\,(\sigma \otimes \id)\, (\id\otimes\alpha)\,\lambda(a)&&\text{}\\
&=
(\id\otimes \id \otimes \alpha)\Big(\overline{\ww}^{\mathbb{H}}_{12}\big((\id\otimes\lambda) \alpha(a)\big)(\overline{\ww}^{\mathbb{H}}_{12})^*\Big)&&\text{by \eqref{YD}}\\
&=
\overline{\ww}^{\mathbb{H}}_{12}\,\Big((\id\otimes \id \otimes \alpha)(\id\otimes\lambda) \alpha(a)\Big)\,(\overline{\ww}^{\mathbb{H}}_{12})^*&&\text{}\\
&=
\overline{\ww}^{\mathbb{H}}_{12}\,\Big((\id\otimes\tilde\lambda)(\id\otimes \alpha)\, \alpha(a)\Big)\,(\overline{\ww}^{\mathbb{H}}_{12})^*&&\text{by \eqref{eq:alf-lamb}}\\
&=
\overline{\ww}^{\mathbb{H}}_{12}\,\Big((\id\otimes \tilde{\lambda}) \, \tilde{\alpha} \,(\alpha(a))\Big)\,(\overline{\ww}^{\mathbb{H}}_{12})^* .&&\text{}
\end{align*}
And, for $\hh x\in L^\infty(\hh\G)$ we have
\begin{align*}
(\sigma\otimes \id)\,(\id\otimes \tilde{\alpha})\, \tilde{\lambda}(\hh x\otimes 1)
&=
\big[(\sigma\otimes \id)\,\big((\ww^{\G}_{23})^* (\overline{\ww}^{\hh\h}_{13})^* (1\otimes 1 \otimes \hh x )\overline{\ww}^{\hh\h}_{13}\ww^{\G}_{23}\big)\big] \otimes 1
\\&= 
\big[(\sigma\otimes \id)\,\big((\overline{\ww}^{\hh\h}_{12})^* (\overline{\ww}^{\hh\h}_{13})^*(\ww^\G_{23})^*(1\otimes 1 \otimes \hh x )\ww^\G_{23} \overline{\ww}^{\hh\h}_{13}\overline{\ww}^{\hh\h}_{12}\big)\big] \otimes 1
 \\&= 
 \big[\overline{\ww}^{\H}_{12}(\overline{\ww}^{\hh\H}_{23})^*
(\ww^\G_{13})^*(1\otimes 1 \otimes \hh x )\ww^\G_{13} \overline{\ww}^{\hh\H}_{23} (\overline{\ww}^{\H}_{12})^*\big]
 \otimes 1
 \\&= \overline{\ww}^{\H}_{12}\Big([(\id\ot\rho)\big((\ww^\G)^*(1 \otimes \hh x )\ww^\G\big)]\otimes 1\Big) (\overline{\ww}^{\H}_{12})^*
 \\&= 
 \overline{\ww}^{\H}_{12}\Big((\id \otimes \tilde{\lambda}) [(\ww^\G)^* (1\otimes \hh x) \ww^\G \otimes 1]\Big) (\overline{\ww}^{\H}_{12})^* 
\\&= 
\overline{\ww}^{\H}_{12}\big((\id\otimes \tilde{\lambda}) \, \tilde{\alpha} \,(\hh x\otimes 1)\big) (\overline{\ww}^{\H}_{12})^* ,
\end{align*}
where we used the definitions of $\tilde\alpha$ and $\tilde\lambda$ in the first equality, we used \eqref{bi-ch} in the second, definition of $\rho$ in the forth equality, and \eqref{eq*} in the fifth equality, and once again the definition of $\tilde\alpha$ in the last one.
This completes the proof.%
\end{proof}
\end{lemma}


\begin{remark}\label{rem3}
\normalfont{In the setup of Lemma~\ref{lem:GHcrossed}, recall that the restriction of $\tilde\alpha$ to $\alpha(M)$ is $(\id\otimes\alpha)$, that is $\tilde\alpha\circ\alpha = (\id\otimes\alpha)\circ\alpha$. This, together with \eqref{eq:alf-lamb} entails that the map $\alpha: M\to \G\ltimes_\alpha M$ is both $\G$-equivariant and $\hh\h$-equivariant. }
\end{remark}

\begin{lemma}\label{lem:act-ext}
Let $\G$ be a discrete quantum group of Kac type and let $\hh\h$ be a closed quantum subgroup of $\hh\G$, and let $M$ be a $(\G,{\H})$-Yetter-Drinfeld von Neumann algebra equipped with left actions $\alpha$ of $\G$ and $\lambda$ of $\hh\h$.	
Then the canonical conditional expectation $E:\G\ltimes_\alpha M \to \alpha(M)$ is both $\G$-equivariant and $\hh\h$-equivariant. 
\end{lemma}
\begin{proof}
By Lemma~\ref{lem:cond-eq} $E$ is $\G$-equivariant. We show it is $\hh\h$-equivariant, and for this, we only need to prove that the restriction of $E$ to $L^\infty(\hh\G)\ot\C$ is $\hh\h$-equivariant. 

Let $\hh x\in C(\hh\G)$, and let ${\hh y}\in C^u(\hh\G)$ be such that $\Lambda_{\hh\G}(\hh y)=\hh x$. On the one hand we have 
\[
\tilde{\lambda}\left( E(\hh x\otimes 1) \right) = (\id_{\G\ltimes_\alpha M}\otimes \varphi^{\hh\G})(\widehat{\alpha}(\hh x\otimes 1))= 
(\id\otimes\id\otimes \varphi^{\hh\G})\Delta^{\hh\G}(\hh x)_{13}
=\varphi^{\hh\G}(\hh x) 1\otimes 1\otimes 1,
\]
and on the other hand, 
\begin{align*}
(\id\otimes E) (\tilde{\lambda} (\hh x\otimes 1))
&= (\id\otimes E) (\rho(\hh x)\otimes 1)
=\big((\id\otimes \varphi^{\hh\G})\rho(\hh x)\big) \otimes 1\\
 &=\big((\id\otimes \varphi^{\hh\G}) ((\Lambda_{\hh\h}\circ \pi)\otimes \Lambda_{\hh\G})\Delta^u_{\hh\G}(\hh y)\big)\otimes 1
 \\  &
 =\big(((\Lambda_{\hh\h}\circ \pi)\otimes \varphi^{\hh\G}\circ \Lambda_{\hh\G})\Delta^u_{\hh\G}(\hh y)\big)\otimes 1
 \\ &
 =\varphi^{\hh\G}( \Lambda_{\hh\G}(\hh y))1\otimes 1\otimes 1
 =\varphi^{\hh\G}(\hh x) 1\otimes 1\otimes 1 ,
\end{align*}
where we used~\eqref{eq*} in the first equation, \eqref{eqq} in the third, and the fact that $\varphi^{\hh\G}\circ \Lambda_{\hh\G}$ is the Haar state on $C^u(\hh\G)$ (\cite[Propositions 8.2 and 8.4]{Kust01}) in the fifth equation.

Since $C(\hh\G)$ is weak$^*$-dense in $L^\infty(\hh\G)$	and $E$ is normal, it follows that $E$ is $\hh\h$-equivariant on $L^\infty(\hh\G)\ot\C$.
\end{proof}

\section{Strongly Approximately Transitive (SAT) State}\label{sec:SAT}

In this section we define SAT states in the setting of lcq group actions, and prove some of their main general properties. This is the natural generalization of the commutative notion, which was introduced by Jaworski in \cite{Jaw}.

\begin{definition}
Let $\alpha$ be an action of a lcq group $\G$ on a von Neumann algebra $N$. A normal state $\nu$ on $N$ is called \emph{strongly approximately transitive (SAT)}, if the norm closure of the set $\{(\omega\otimes \nu) \alpha\st\omega ~{\text{is a normal state on}} ~\LL \}$ contains all normal states on $N$.
\end{definition}
For $\alpha$ and $\nu$ as in the above definition, the \emph{Poisson map} $P_\nu:N\to\LL$ is defined by $P_\nu(x)=(\id\otimes\nu)\alpha(x)$.

The following is a useful characterization of SAT states. The proof  is similar to \cite[Lemma 4.2]{KKSV} (also, cf. \cite[Proposition~2.2]{Jaw}). We include the proof for the convenience of the reader. 
\begin{proposition}\label{isometry}
Let $\alpha$ be an action of a lcq group $\G$ on a von Neumann algebra $N$. A normal state $\nu$ on $N$ is SAT if and only if the Poisson map $P_\nu$ is isometric on the self-adjoint part of $N$.
\end{proposition}
\begin{proof}
Let $\nu$ be a SAT state. For a given self-adjoint element $x\in N$ and $\varepsilon>0$, there is a normal state $\rho$ on $N$ such that $|\rho(x)|>\|x\|-\varepsilon$. 
Since $\nu$ is a SAT state, there exists a normal state $\omega$ in $\LO$ such that 
$\|x\|-\varepsilon< |(\omega\otimes\nu)\alpha(x)|=|\,\omega(P_\nu(x))|$.
Since $\varepsilon$ was arbitrary and $P_\nu$ is a contraction, we have $\|P_\nu(x)\|=\|x\|$.
	
Conversely, assume $\nu$ is not SAT, and let $\rho$ be a normal state on $N$ that does not belong to the norm closure of $\{(\omega\otimes \nu) \alpha\st\omega ~{\text{is a normal state on}} ~\LL \}$.

Then by the Geometric Hahn-Banach theorem, there exists $x\in N$, $r\in \mathbb R$ and $\varepsilon > 0$ such that ${\rm Re}(\omega(P_\nu(x)))= {\rm Re}((\omega\otimes\nu)\alpha(x))\leq r < r+\varepsilon\le {\rm Re}(\rho(x))\leq \|x\|$ for every state $\omega\in \LO$. Note that the above inequalities hold the same if we replace $x$ by $x^*$, and therefore we may assume $x$ is self-adjoint.
Then, $P_\nu(x)$ is a self-adjoint element in $N$, and therefore $\|P_\nu(x)\|=\sup\left\{|\omega\left(P_\nu(x)\right)|:\omega~{\text{is a normal state on}} ~\LL\right\}\le \rho(x) - \varepsilon \le \|x\|- \varepsilon <\|x\|$, which contradicts $P_\nu$ being isometric.
\end{proof}
In particular, Proposition~\ref{isometry} provides us with a source of natural examples of SAT states, namely, 
\emph{noncommutative Poisson boundaries} in the sense of Izumi \cite{I}, defined as follows (see \cite{KNR1} for the general locally compact case).

Given a discrete quantum group $\G$ and a normal state $\mu\in \ell^1(\G)$, we denote $H^\infty(\G,\mu)=\{x\in \ell^\infty(\G) \st (\id\otimes\mu)\Delta^\G(x)=x\}$ for the space of all \emph{$\mu$-harmonic} elements in $\ell^\infty(\G)$. The space $H^\infty(\G,\mu)$ is a weak* closed operator subsystem of $\ell^\infty(\G)$ which admits a canonical multiplication, turning it into a von Neumann algebra, called the noncommutative Poisson boundary of the pair $(\G, \mu)$. The co-multiplication $\Delta^\G$ restricts to an action of $\G$ on $H^\infty(\G,\mu)$, turning it into a $\G$-von Neumann algebra \cite[Proposition 2.1]{KNR2}. 

In this case, the restriction of the unit $\varepsilon^\G$ of $\DO$ to $H^\infty(\G,\mu)$ is a SAT state. Indeed, the map $(\id\otimes\varepsilon^\G)\Delta^\G \colon H^\infty(\G,\mu)\to \DD$ is just the inclusion, hence isometric. We also observe that the restrictions of $\varepsilon^\G$ and $\mu$ to ${H^\infty(\G,\mu)}$ coincide.

We will use in several places the fact that the crossed product $\G\ltimes_{\Delta^\G} H^\infty(\G,\mu)$ is an injective von Neumann algebra \cite[Corollary 2.5]{KNR2}.
\\[1ex]
Recall that an action $\alpha$ of a lcq group $\G$ on a von Neumann algebra $N$ is said to be \emph{ergodic} if $N^\alpha=\mathbb{C}1_N$, where $N^\alpha:=\{x\in N\st \alpha(x)=1\otimes x\}$ is the fixed point algebra of the action. 
\begin{proposition}
Let $\alpha$ be an action of a lcq group $\G$ on a von Neumann algebra $N$ and suppose that $N$ admits a SAT state. Then $\alpha$ is ergodic.
\end{proposition}

\begin{proof}
Let $\nu$ be a SAT state. For every $x\in N^\alpha$ we have $(\id\otimes\nu)\alpha(x) = \nu(x) 1$. Since $N^\alpha$ is a von Neumann subalgebra of $N$, it is the span of its self-adjoint part, thus, it follows from Proposition \ref{isometry} that $N^\alpha=\mathbb{C}1$.
\end{proof}

It is obvious that if a $\G$-von Neumann algebra $N$ admits a SAT state which is also $\G$-invariant, then $N=\mathbb{C}$ is trivial. Indeed, even the co-existence of a SAT and an invariant state on $N$ impose a strong structural restriction on $N$.
In the case of locally compact group actions, it was shown in \cite[Proposition 2.6]{Jaw} that any SAT measure on a measurable $G$-space that is absolutely continuous with respect to a $\sigma$-finite invariant measure, is purely atomic. In particular, it follows for instance that if $G\act (X, \mu)$ is an ergodic probability measure preserving action, and there is a SAT measure $\nu$ in $L^1(X, \mu)$, then $X$ must be a finite space
(hence, the action of $G$ on $X$ is equivalent to the action of $G$ on the coset space of a finite-index subgroup).

We conclude this section with a generalization of the above-mentioned result of \cite{Jaw} to the quantum case. 
Recall that a von Neumann algebra $N$ is called {\it purely atomic} if every projection in $N$ has a minimal subprojection.

\begin{theorem}
Let $\alpha$ be an action of a lcq group $\G$ on a von Neumann algebra $N$. Suppose that $N$ admits an $\alpha$-invariant 
n.s.f. tracial weight and a faithful SAT state. Then $N$ is purely atomic.
\end{theorem}
\begin{proof}
Let $\tau$ be an $\alpha$-invariant n.s.f. tracial weight and $\nu$ be a  faithful SAT state on $N$. 

We claim that there is $\varepsilon>0$ such that $\nu(x)<\frac{1}{2}$ for all positive $x$ in the unit ball of $N$ with $\tau(x)<\varepsilon$. Indeed, since $\tau$ is a trace, the space $\{\tau(a \,\cdot) : a\in N\}$ is dense in the predual $N_*$ \cite{Dix53}, and therefore since $\nu$ is normal, there is a positive element $a\in N$ such that $\tau(a)<\infty$ and $\|\tau(a \,\cdot) - \nu(\cdot)\|_{N_*} <\frac14$. Let $\varepsilon=\frac{1}{4\|a\|}$. If $0\leq x\leq 1$ and $\tau(x)<\varepsilon$, then
$\nu(x) = \nu(x) - \tau(a x) + \tau(a x) \le \frac14 + \|a\| \tau(x) <\frac12 ,$
and the claim follows.

Now, for the sake of contradiction, assume $N$ is not purely atomic, and let $q\in N$ be a non-zero projection with no minimal subprojections. Since $\tau$ is semifinite, by passing to a subprojection, we may further assume that $\tau(q)<\infty$. Then $q$ contains a non-zero projection $p$ such that $\tau(p)<\varepsilon$. To see this, note that $qMq$ is a finite von Neumann algebra. If its center $Z(qMq)$ is not purely atomic, then it contains a non-zero projection $q_0$ such that $q_0Z(qMq)q_0$ contains no minimal projection, hence is the $L^\infty$ of a probability space $(X, \mu)$ with no atoms (where $\mu$ is defined by the restriction of $\tau$), and in particular containing measurable subsets of arbitrary small non-zero measures. This implies $q_0Z(qMq)q_0$ contains non-zero projections $p$ with $\tau(p)<\varepsilon$. If $Z(qMq)$ is purely atomic, then $qMq$ is a direct sum of finite factors, which have to be type ${\rm II}_1$ since $qMq$ does not contain any minimal projections. It is a well-known fact that every ${\rm II}_1$ factor contains projections with arbitrary non-zero small trace values (see e.g. \cite[Proposition 4.1.6]{DP}).
In either case, $q$ contains a non-zero projection $p$ such that $\tau(p)<\varepsilon$. 
Let $\rho(\cdot) = \frac{1}{\tau(p)}\tau(p\,\cdot)$. Then $\rho(1)=1$, and therefore $\rho$ is a normal state on $N$.
Since $\nu$ is SAT, there is a state $\om\in L^1(\G)$ such that
$| \rho(p) - (\om\otimes \nu)\alpha(p) | < \frac12$,
which implies $\nu\big((\om\otimes \id)\alpha(p)\big) > \frac12$.
On the other hand, since $\tau$ is $\alpha$-invariant, we have
$\tau\big((\om\otimes \id)\alpha(p)\big) = \tau(p)< \varepsilon$.
This contradicts the fact we established at the beginning of the argument, and hence completes the proof.
\end{proof}

\section{Unique stationary SAT states}\label{uss}

The importance of unique stationarity in 
measurable boundaries was known from the beginning of the theory. 
The study of these systems was initiated by Furstenberg in~\cite{Furs}, where they were called \emph{$\mu$-proximal actions}. They were further studied in~\cite{Mar, GW}.
Unique stationarity in the setting of group actions on $C^*$-algebras was studied in \cite{HK1}, where a characterization of $C^*$-simplicity was proved in terms of unique stationarity of the canonical trace.
So, it should not come as a surprise that actions admitting unique stationary SAT states behave similarly. 
This section is concerned with such actions of discrete quantum groups. 

Let $\G$ be a lcq group, let $\mu\in C_0(\G)^*$ be a state, and let $\alpha$ be an action of $\G$ on a unital $C^*$-algebra $\sA$. A state $\nu$ on $\sA$ is called \emph{$\mu$-stationary} if $(\mu\otimes\nu)\alpha=\nu$. 
A standard fixed-point argument implies 
that for every state $\mu\in C_0(\G)^*$, the set of $\mu$-stationary states on $\sA$ is non-empty; indeed, for every state $\nu$ on $\sA$, $(\mu\otimes\nu)\alpha\in C_0(\G)^*$ is also a state, and the map $\nu\mapsto (\mu\otimes\nu)\alpha$ is a contractive continuous affine map on the compact convex space of states on $\sA$, hence has a fixed point by Markov--Kakutani fixed point theorem. Obviously any such fixed point is a $\mu$-stationary state on $\sA$.


\begin{lemma}\label{lem:rigid}
Let $\G$ and $\H$ be lcq groups, let $N$ be a von Neumann algebra equipped with a (left or right) action $\alpha$ of $\G$ and a (left or right) action $\beta$ of $\H$, and let $\sA$ be both $\G$-invariant and $\H$-invariant unital $C^*$-subalgebra of $N$.

Assume that for some state $\mu\in\LO$, $N$ admits a $\mu$-stationary normal $\G$-SAT state $\nu$ 
such that the restriction of $\nu$ to $\sA$ is the unique $\mu$-stationary $\H$-invariant state on $\sA$.

Then, every ucp map $\Phi:N\to N$ that is both $\G$-equivariant and $\H$-equivariant, restricts to identity on $\sA$.
\end{lemma}

\begin{proof}
Let $\Phi:N\to N$ be ucp and both $\G$-equivariant and $\H$-equivariant, and $\Phi^*: N^*\to N^*$ its adjoint map. Since $\Phi$ is $\G$-equivariant, by $\mu$-stationarity of $\nu$, for every $a\in \sA$ we get
\begin{align*}
(\mu\otimes\Phi^*(\nu))\alpha(a)
&=\nu\big(\Phi((\mu\otimes\id)\alpha (a))\big)
=\nu\big((\mu\otimes\id)\alpha (\Phi(a))\big)
\\&= (\mu\otimes\nu)\alpha(\Phi(a))=\nu(\Phi(a))) = \Phi^*(\nu)\big((a)\big),
\end{align*}
which shows that the restriction of $\Phi^*(\nu)$ to $\sA$ is $\mu$-stationary. 
Arguing similarly, we also see that $\Phi^*(\nu)$ restricts to an $\H$-invariant state on $\sA$.
By the uniqueness assumption, the restrictions of $\Phi^*(\nu)$ and $\nu$ to $\sA$ coincide.
Thus, for every $\omega\in \LO$ and $a\in \sA$ we have
\begin{align*}
&[\Phi^*((\omega\otimes \nu)\alpha)](a)=(\omega\otimes \nu)\alpha(\Phi(a))
=(\omega\otimes \Phi^*(\nu))\alpha(a)\\=\ &\Phi^*(\nu)\big((\omega\otimes \id)\alpha(a)\big)=\nu\big((\omega\otimes \id)\alpha(a)\big)=((\omega\otimes \nu)\alpha(a)) ,
\end{align*}
where in the second equality we used $\G$-equivariance of $\Phi$, and in the forth equality we used the fact that $(\omega\otimes \id)\alpha(a)\in \sA$ which follows from $\G$-invariance of $\sA$.
Since $\nu$ is $\G$-SAT, the space $\{(\omega\otimes \nu)\alpha \st \om\in\LO\}$ is norm-dense in $N_*$, and therefore,
$\rho(\Phi(a)) = \rho(a)$ for all $\rho\in N_*$ and $a\in \sA$. Since the restrictions to $\sA$ of functionals in $N_*$ form a weak* dense subset of $\sA^*$, we conclude $\Phi(a)=a$ for all $a\in \sA$.
\end{proof}


The above lemma generalizes \cite[Theorem 3.4]{HK}, where a similar result proved for measurable boundary actions with unique stationary compact models. The proof of the latter is based on a result of Margulis \cite[Corollary 2.10(a)]{Mar}, which allows to pass from morphisms at the level of function algebras to maps at the level of underlying sets, and then use rigidity properties of boundaries. Such an argument is obviously not very quantizable: there are no underlying sets involved. Our result above shows that it is indeed the SAT property of the boundary that is behind the above rigidity. And the proof is the manifestation of our remarks in the introduction that the operator theoretic nature of the SAT property can be a significant advantage compared to boundary actions in the noncommutative setting.

Let $\G$ be a lcq group and $M$ be a (left or right) $\G$-von Neumann algebra. We say $M$ is $\G$-injective if for every $\G$-von Neumann algebra $N$ containing $M$ as a $\G$-invariant von Neumann subalgebra, there is a $\G$-equivariant conditional expectation $P: N\to M$.

Let $M\subset N$ be an inclusion of von Neumann algebras.
By an \emph{intermediate von Neumann subalgebra} of the inclusion 
we mean a von Neumann algebra $Q$ such that $M\subset Q\subset N$. We say $Q$ is proper if $Q\neq N$.

\begin{theorem}\label{main}
Let $\G$ be a discrete quantum group of Kac type, let $\hh{\mathbb{H}}$ be a closed quantum subgroup of $\hh\G$, and let $N$ be a $(\G,\hh{\mathbb{H}})$-Yetter-Drinfeld von Neumann algebra equipped with left actions $\alpha$ of $\G$ and $\lambda$ of $\hh\h$. 
Assume that $N$ admits a $\G$-SAT state $\nu$ and a weak* dense unital $C^*$-subalgebra $\sA$ such that
\begin{itemize}
\item[(i)\,\,] 
$\nu$ is $\hh\H$-invariant;
\item[(ii)\,] 
$\nu$ is $\mu$-stationary for some state $\mu\in L^1(\G)$;
\item[(iii)]
$\sA$ is both $\G$-invariant and $\hh{\H}$-invariant;
\item[(iv)]
the restriction of $\nu$ to ${\sA}$ is the unique $\mu$-stationary $\hh{\H}$-invariant state on $\sA$.
\end{itemize}
Then the inclusion $\LLL\subset \G\ltimes_{\alpha} N$ does not admit any proper ${\hh{\mathbb{H}}}$-injective intermediate von Neumann subalgebras. 
\end{theorem}
\begin{proof}
Assume $M$ is an ${\hh{\mathbb{H}}}$-injective von Neumann algebra with $\LLL\subseteq M \subseteq \G\ltimes_{\alpha} N$. We will show $M=\G\ltimes_{\alpha} N$. By Lemma~\ref{lem:GHcrossed} the crossed product $\G\ltimes_{\alpha} N$ is a $(\G,{\H})$-Yetter-Drinfeld von Neumann algebra, and $\alpha(N)$ is both $\G$ and $\hh\h$-invariant in $\G\ltimes_{\alpha} N$. 
By ${\hh{\mathbb{H}}}$-injectivity of $M$, there exists an ${\hh{\mathbb{H}}}$-equivariant conditional expectation 
$P:\G\ltimes_{\alpha} N\to M$.

Since $\LLL\subseteq M$ and
$W\in \DD\vtp \LLL$, it follows that $W_{12}$ is in the multiplicative domain of $\id\otimes P$, hence
$(\id\otimes P)\tilde\alpha(y)=(\id\otimes P)(W^*_{12}\, y_{23}\, W_{12})
=
W^*_{12}\,(P(y)_{23})\, W_{12}
=
\tilde\alpha(P(y))$
for all $y\in \G\ltimes_{\alpha} N$, 
where $\tilde\alpha$ is the action of $\G$ on $\G\ltimes_{\alpha} N$ defined in \eqref{beta}.
This implies that $M$ is $\G$-invariant and $P$ is $\G$-equivariant. 

Let $\sA\subset N$ be as in the assumptions. 
Recall that $E:\G\ltimes_{\alpha}N\to \alpha(N)$ denotes the canonical conditional expectation defined in Section~\ref{sect2}.
%
The map $E\circ P|_{\alpha(\sA)}\colon\alpha(\sA)\to \alpha(N)$ is ucp and both $\G$-equivariant and $\hh{\mathbb{H}}$-equivariant, since both $P$ and $E$ are such, by the above and Lemma~\ref{lem:act-ext}. 
By Remark~\ref{rem3}, the map $\Phi:=\alpha^{-1}\circ E\circ P\circ \alpha\colon N\to N$ is also $\G$-equivariant and $\hh{\mathbb{H}}$-equivariant, hence restricts to identity by Lemma~\ref{lem:rigid}.
%
We conclude that 
$E\circ P$ restricts to identity on $\alpha(\sA)$.
Using the module property of the conditional expectation $E$, and applying the Schwartz inequality to the ucp map $P$, we get for all $x\in\alpha(\sA)$,
\[\begin{split}
E\big((x-P(x))^*(x-P(x))\big) &= E\big((x^*x-P(x)^*x-x^*P(x)+P(x)^*P(x)\big) \big) \\&\le E(x^*x)-E(P(x^*))x-x^*E(P(x))+E(P(x^*x)) \\&= x^*x-x^*x-x^*x+x^*x = 0 .
\end{split}\]
Since $E$ is faithful, it follows that $P$ restricts to the identity map on $\alpha(\sA)$.
Hence, $\alpha(\sA)\subseteq M$, and therefore $\alpha(N)\subseteq M$ since $\sA$ is weak* dense in $N$, and $M$ is weak* closed.

We also have $\LLL\subseteq M$ by the assumption. Since $\G\ltimes_{\alpha} N$ is generated by $\LLL$ and $\alpha(N)$, it follows $M=\G\ltimes_{\alpha} N$.
\end{proof}

\section{Examples and applications}\label{ex}
In this section we apply Theorem~\ref{main} to certain classes of discrete quantum groups for which concrete realization of Poisson boundaries with uniquely stationary models have been obtained.

First, let us recall some standard facts about noncommutative Poisson boundaries of discrete quantum groups $\G$. 
We denote the set of all equivalence classes of  irreducible representations of a compact quantum group $\hh\G$ by $\irr(\hh\G)$. For every $s\in\irr(\hh\G)$ we denote by $H_s$ the corresponding Hilbert space, by $U_s\in B(H_s)\otimes C(\hh\G)$ the unitary corepresentation and by $\pi_s:\DD\to B(H_s)$ the corresponding representation of $\DD$. 
There is a unique state $\phi_s:\DD\to\CC$ satisfying 
$\phi_s(x)1_{H_s} = (\id\otimes \varphi^{\hh\G})(U^*_s(\pi_s(x)\otimes 1) U_{s})$
for all $x\in\DD$ (see e.g. \cite[Notation 1.7 and Notation 1.11]{VV}). 
Let $\mu$ be a  probability measure on $\irr(\hh\G)$, and define 
$\phi_\mu:=\sum_{s\in\irr{\hh\G}}\mu(s)\phi_s$.
Then $\phi_\mu$ is a $\hh\G$-invariant normal state on $\DD$, and the map 
$P_{\phi_\mu}:\DD\to\DD$ defined by $x\mapsto (\id\otimes\phi_\mu)\Delta^\G(x)$, restricts to a Markov operator on $\ell^\infty(\irr(\hh\G))$, hence yields a classical random walk on the set $\irr(\hh\G)$.  The probability measure $\mu$ is said to be generating if the corresponding random walk on $\irr(\hh\G)$ is irreducible, and it is said to be ergodic if the Poisson boundary of the random walk is trivial.

It is proved in \cite[Lemma 3.5]{HHN22}, that for ergodic generating probability measures $\mu$ on $\irr(\hh\G)$, the Poisson boundary $H^\infty(\G,\phi_\mu)$ is independent of the choice of $\mu$, and in fact, 
\[
H^\infty(\G,\phi_\mu)= H^\infty(\G):=\{ x\in\DD \st (\id\otimes\phi_s)\Delta^\G(x) = x \ \, \text{for all} \ \, s\in \irr(\hh\G)\} .
\] 
We refer the reader to \cite{HHN22} for the relevant definitions and further details.

\subsection{Orthogonal free quantum groups} 

First, we consider Van Daele and Wang's orthogonal free discrete quantum groups $\G = \FO_N$, the quantum group duals of the free orthogonal compact quantum groups $O_N^+=\hh\G$.

The quantum groups $\FO_N$ are of Kac type.
In \cite{VV} a Gromov boundary $C^*$-algebra was constructed for $\FO_N$, which was then shown in \cite{KKSV}, by means of a unique stationarity result, to be also a topological boundary in the sense of \cite[Definition 4.1]{KKSV}.



Let $M\subset N$ be an inclusion of von Neumann algebras. We say $N$ is a {minimal injective extension} of $M$, if $N$ is injective and no proper intermediate von Neumann subalgebra of the inclusion is injective.  We say $M$ is a {maximal injective} in $N$, if $M$ is injective and every von Neumann subalgebra of $N$ that contains $M$ properly is not injective.

\begin{theorem}\label{thm:ex-free-qg}
Let $\G = \FO_N$ be Van Daele and Wang's orthogonal free discrete quantum group with $N\geq 3$. Let $\theta$ be an n.s.f. weight on $H^\infty(\G)$, $H_\theta$ its GNS Hilbert space, and 
$U_\theta\in B(\ell^2(\G)\otimes H_\theta)$ the unitary implementation of the action of $\G$ on $H^\infty(\G)$. 
Then the following hold.
\begin{enumerate}
\item
$\G\ltimes_{\Delta^{\G}} H^\infty(\G)$ is a minimal injective extension of $L^\infty(O^+_N)$.
\item
$U_\theta\,\big(\G^{\rm op}\ltimes_{({\Delta^{\G})}^\prime} H^\infty(\G)^\prime\big)\,U_\theta^* $ is a maximal injective in $L^\infty(O^+_N)^\prime\vtp B(H_\theta)$. 
\end{enumerate}
\end{theorem}
\begin{proof}
(1)\ As remarked in Section~\ref{sec:SAT}, the crossed product von Neumann algebra $\G\ltimes_{\Delta^{\G}} H^\infty(\G)$ is injective. Thus, it suffices to check that the assumptions of Theorem~\ref{main} hold for the action of $\G$ on $H^\infty(\G)$ and the trivial choice of $\hh{\H}=\mathbb{C}$.

Let $\B_\infty$ be the Gromov boundary of $\G$ in the sense of Vaes--Vergnioux~\cite{VV}, which is a unital $\G$-$C^*$-algebra. 
Let $\mu$ be an ergodic generating probability measure on $\irr(O^+_N)$.
By \cite[Theorem 7.2]{KKSV} there is a unique $\mu$-stationary state $\nu$ on $\B_\infty$, and by \cite[Theorem 5.6]{VV} there is a canonical $\G$-von Neumann isomorphism $\pi_\nu(\B_\infty)^{\prime\prime} \cong H^\infty(\G)$. Since the restriction $\mu\!\!\mid_{H^\infty(\G)}$ is SAT, $\nu$ is SAT as a normal state on $\pi_\nu(\B_\infty)^{\prime\prime}$. Thus, the conditions of Theorem~\ref{main} hold for the $\G$-von Neumann algebra $\pi_\nu(\B_\infty)^{\prime\prime}$, and consequently for $H^\infty(\G)$. Hence, the first assertion follows.
\\[1ex]
(2)\ Let
\begin{equation}\label{eq:vn-incls}
U_\theta(\G^{\rm op}\ltimes_{({\Delta^{\G})}^\prime} H^\infty(\G)^\prime)U_\theta^* \subseteq M \subseteq L^\infty(O^+_N)^\prime\vtp B(H_\theta)
\end{equation}
be inclusions of von Neumann algebras, and assume that $M$ is injective. By Theorem~\ref{commutant} the commutant of $\G\ltimes_{\Delta^{\G}} H^\infty(\G)$ in $B(L^2(O^+_N) \otimes H_\theta)$ is $U_\theta(\G^{\text{op}}\ltimes_{({\Delta^{\G})}^\prime} H^\infty(\G)^\prime)U_\theta^*$. Hence, taking commutants inside $B(L^2(O^+_N) \otimes H_\theta)$ of the von Neumann algebras in the inclusions \eqref{eq:vn-incls}, we get $L^\infty(O^+_N) \subseteq M^\prime \subseteq \G\ltimes _{\Delta^{\G}} H^\infty(\G)$, and $M^\prime$ is injective. Now part (1) above implies $M^\prime = \G\ltimes_{\Delta^{\G}} H^\infty(\G)$. Taking commutant in $B(L^2(O^+_N) \otimes H_\theta)$ once again, we conclude that $M = U_\theta(\G^{\text{op}}\ltimes_{{(\Delta^{\G})}^\prime} H^\infty(\G)^\prime)U_\theta^*$.
\end{proof}

\subsection{$\hh\G$-injectivity}
We prove a rigidity property for $\hh\G$-injective extensions of $\LLL$, similar to part (1) of Theorem~\ref{thm:ex-free-qg}, but for a larger class of discrete quantum groups $\G$.

\begin{theorem}\label{thm:Drinf-doub}
Let $\G$ be a discrete quantum group of Kac type. Assume that $\irr(\hh\G)$ admits a generating ergodic probability measure. Then the crossed product $\G\ltimes_{\Delta^{\G}} H^\infty (\G)$ is the minimal $\hh\G$-injective von Neumann algebra extension of $\LLL$.
\end{theorem}
\begin{proof}
As noted in \cite[Section 3]{HHN22}, the von Neumann algebra $H^\infty (\G)$ is a Yetter-Drinfeld $\G$-von Neumann algebra, that is a $(\G,\hh{\G})$-Yetter-Drinfeld von Neumann algebra in the terminology of our Definition~\ref{defYD}.
The fact that $\G\ltimes_{\Delta^{\G}} H^\infty (\G)$ is $\hh\G$-injective follows from \cite[Theorem 5.4 \& Corollary 7.7]{MO}. 
We check that the conditions of Theorem~\ref{main} hold for the above actions of $\G$ and $\hh{\G}$ on $H^\infty(\G,\mu)$.
Let $\mu\in {\rm Prob}({\rm Irr}(\hh\G))$ be generating and ergodic. 
The combination of \cite[Proposition 3.3, Theorem 3.7, Theorem 3.10]{HHN22} implies that $H^\infty(\G)$ 
contains a weak* dense unital $C^*$-subalgebra $\sA$ that is both $\G$ and $\hh{\G}$-invariant, and with the property that $\sA$ admits a unique $\hh\G$-invariant $\mu$-stationary state $\nu$ such that there is a canonical $\G$-von Neumann isomorphism $\pi_\nu(\sA)^{\prime\prime} \cong H^\infty(\G)$. %
Thus, the conditions of Theorem~\ref{main} hold, and we conclude $\G\ltimes_{\Delta^{\G}} H^\infty (\G)$ is the minimal $\hh\G$-injective von Neumann algebra extension of $\LLL$.
%
\end{proof}

Examples of discrete quantum groups $\G$ for which $\irr(\hh\G)$ admits a generating ergodic probability measure include those where $\hh\G$ is monoidally equivalent to a coamenable compact quantum group (see the remarks on pages 327 and 332 of \cite{HHN22}).
For concrete example of a discrete Kac algebra with this property, beyond the case of orthogonal free discrete quantum groups, see the recent paper \cite{VR24}.

\end{document}